\documentclass{amsart}
\usepackage{amsmath, amsfonts}
\usepackage{amssymb}
\usepackage{mathtools}
\usepackage{tikz-cd}

\usepackage{amsthm}
\usepackage{upgreek}
\usepackage{hyperref}
\usepackage{float, textcomp}
\usepackage{hyperref}
\usepackage{chngpage}
\usepackage{graphicx}
\usepackage{float}
\usepackage{caption}
\usepackage{cleveref}

\theoremstyle{plain}
\newtheorem{theorem}{Theorem}

\newtheorem{prop}[theorem]{Proposition}
\newtheorem{lemma}[theorem]{Lemma}
\newtheorem{corollary}[theorem]{Corollary}

\theoremstyle{definition}
\newtheorem{definition}[theorem]{Definition}

\numberwithin{theorem}{section} 

\theoremstyle{remark}
\newtheorem*{remark}{Remark}

\DeclareMathOperator{\Tr}{Tr}

\DeclareMathOperator{\Gal}{Gal}

\DeclareMathOperator{\Sym}{Sym}
\DeclareMathOperator{\vspan}{span}
\DeclareMathOperator{\sgn}{sgn}

\DeclareMathOperator{\lcm}{lcm}

\newcommand{\Fp}{\mathbb{F}_p}

\newcommand{\Fq}{\mathbb{F}_q}

\newcommand{\Cp}{\mathbb{C}_p}

\newcommand{\Qp}{\mathbb{Q}_p}
\newcommand{\Qq}{\mathbb{Q}_q}
\newcommand{\Qqk}{\mathbb{Q}_{q^k}}
\newcommand{\Zp}{\mathbb{Z}_p}
\newcommand{\Zq}{\mathbb{Z}_q}
\newcommand{\Zqk}{\mathbb{Z}_{q^k}}

\newcommand{\Zptimes}{\mathbb{Z}_p^\times}

\DeclareMathOperator{\Hom}{Hom}

\DeclareMathOperator{\ordp}{ord}

\newcommand{\ord}[1]{\ordp_p{#1}}

\date{\today}
\title{ Exponential Sums of Witt Towers over Affinoids }
\author{Matthew Schmidt}
\email{mwschmid@buffalo.edu}
\address{Department of Mathematics, SUNY Buffalo}

\usepackage{amsfonts}

\begin{document}

\begin{abstract}
In this paper we construct a Dwork theory for general exponential sums over affinoids in Witt towers. Using this, we compute the degree of the $L$-function, its Hodge polygon and examine when the Hodge and Newton polygons coincide.
\end{abstract}

\maketitle 
\tableofcontents
     
Let $p$ be a prime, $q=p^a$ a $p$-power, and let $m\geq 1$ be an integer such that $p>m$. Write $\Gal(\Qq/\Qp)=\langle\tau\rangle$. Suppose we have:
	$$f(x)=\sum_{i=0}^{m-1}\sum_{j=1}^\ell\sum_{k=0}^{d_{ij}}V^i(a_{ijk}\frac{1}{(x-{P}_j)^k}, 0, \cdots, )\in W_m(\Fq[\frac{1}{x-{P}_1}, \cdots, \frac{1}{x-{P}_1}]),$$
	where $P_1,\cdots, P_\ell$ are distinct elements in $\Fq\cup\{\infty\}$, $W_m(R)$ is the truncated  ring of Witt vectors of some ring $R$, and $V$ is the Witt vector shifting operator. Without loss of generality, take $p\nmid d_{ij}$ for all $i,j$ and suppose that for each $1\leq j\leq \ell$, the maximum $\max_i d_{ij}p^{m-i-1}$ is uniquely achieved. We will assume that $P_1 = 0$ and $P_2=\infty$.   
If $\zeta_{p^m}$ is a primitive $p^m$th root of unity, the exponential sum, $L$-function and characteristic function $C_f(s)$ attached to $f(x)$ are:
\begin{align}
S_f(k) &= \sum_{\substack{x\in\mathbb{F}_{q^k}^\times,\\ x\neq {P}_1, \cdots, {P}_\ell}}\zeta_{p^m}^{\Tr_{W_m(\mathbb{F}_{q^k})/W_m(\Fp)}(f(x))}\\
L_f(k) &= \exp(\sum_{k=1}^\infty S_f(k)\frac{s^k}{k}) \\
C_f(k) &=\exp(\sum_{k=1}^\infty -(q^k-1)^{-1} S_f(k)\frac{s^k}{k}).
\end{align}

In this paper we build on the methods developed in \cite{Schmidt1}, where we constructed a ``universal'' Dwork theory that applies to exponential sums over affinoids. Using the framework from there, we build an alternative dwork cohomology using a truncated Artin-Hasse exponential which allows us to directly compute the degree of the $L$-function. To our knowledge, the general ($m>1$) truncated Artin-Hasse exponential has not been used in this setting before. Furthermore, we compute the Hodge polygon via the traditional Artin-Hasse exponential and use these $p$-adic estimates to generalize Zhu's result from \cite{Zhu} about when the Newton and Hodge bounds coincide.

Our main results are the following two theorems.
\begin{theorem}\label{theorem:main1}
The power series $L_f(s)$ is a polynomial in $s$ of degree 
	\[
		d = (\sum_{j=1}^\ell (\max_{0\leq i\leq m-1} p^{m-i-1}d_{ij}+1))-2.
	\]
\end{theorem}
This result is not new (see Remark 4.7 in \cite{KostersWanPub}), but prior results use the geometry of the Witt tower whereas we utilize $p$-adic methods. Our second result is the computation of the Hodge polygon of $L_f(s)$, which is described as follows.

\begin{theorem}\label{theorem:main2}
Let $\mathrm{NP}_f$ be the $p$-adic Newton polygon of $L_f(s)$. Then $\mathrm{NP}_f$ lies above the polygon with slopes: 
 \[
 	\left \{\left \{ \frac{an}{d_{i_j,j}p^{m-i_j-1}}\right \}_{n=0}^{d_{i_j,j}p^{m-i_j-1}}\right \}_{j=1}^\ell,
 \]	
 where $i_j$ is such that $p^{m-i_j-1}d_{i_j,j}=\max_{0\leq i\leq m-1} p^{m-i-1}d_{ij}$. Furthermore, if $i_j=m-1$ for all $1\leq j\leq \ell$, these polygons coincide if and only if $p\equiv 1\bmod \lcm_j d_{i_j,j}$.
\end{theorem}

\section{Lifting $f(x)$}

We will first lift $f(x)$ to a $p$-adic ring in which we can construct a $p$-adic Dwork cohomology. Recall, that if $a\in W(R)$, we denote by $\bar{a}$ the image of $a$ in $R$, that is $a \bmod p$, and for $a\in R$, we define the Teichm\"uller lift of $a$ in $W(R)$ to be $\hat{a}$.

\begin{lemma}\label{witt_iso_m}
There is a ring isomorphism:
\begin{align*}
	\omega: W_m(\mathbb{F}_{q^k}) &\to \mathbb{Z}_{q^k}/p^m\mathbb{Z}_{q^k}\\
		(x_0, \cdots, x_{m-1})&\mapsto \sum_{i=0}^{m-1}p^i\widehat{x_i^{p^{-i}}} \bmod p^m.
\end{align*}
\end{lemma}
\begin{proof}
It's well known that $\Zqk\cong W(\mathbb{F}_{q^k})$ via the isomorphism 
$$(x_0, x_1,x_2, \cdots)\mapsto \sum_{i=0}^\infty\widehat{x_i^{p^{-i}}}p^i,$$
 and so the lemma follows from  the isomorphism $W_m(K)\cong W(K)/V^mW(K)$.
\end{proof}

Define the additive character $\chi: W_m(\Fp) \to\Cp$ by mapping $\chi(x)= \zeta_{p^m}^{\omega(x)}$ and extend it to $W_m(\mathbb{F}_{q^k})$ by composing it with the trace: 
\[
	\chi_{q^k}(x) = \chi(\omega(\Tr_{W_m(\mathbb{F}_{q^k})/W_m(\Fp)}(x))).
\]
The exponential sum can then be written:
 \begin{align}\label{expo_sum_char}
 	S_f(k)  &= \sum_{\substack{x\in\mathbb{F}_{q^k}^\times,\\ x\neq {P}_1, \cdots, {P}_\ell}}\chi_{q^k}(f(x)).
 \end{align}

Using Lemma~\ref{witt_iso_m} we can lift the exponential sum $S_f(k)$ from an object defined over a finite field as in (\ref{expo_sum_char}), to one defined over a $p$-adic ring. However, before we can choose an appropriate lifting of $f(x)$  some technical lemmas are required.
  
\begin{lemma}\label{lemma:pkpower}
	Let $x,y\in \Cp$ such that $\ord(x)\geq 0$ and $\ord(y)\geq 0$.  Then for $n\geq 1$,
		$(x+py)^{p^n}= x^{p^n} + p^ng(x,y)$.
\end{lemma}
\begin{proof}
Clearly:
\begin{align*}
	(x+py)^{p^n} &= \sum_{r=0}^{p^n} \binom{p^n}{r}x^{n-r}(py)^r.
\end{align*}
When $r>0$, $\ord(\binom{p^n}{r}) = n-\ord(r)$ and hence
\begin{align*}
\ord(\binom{p^n}{r}x^{n-r}(py)^r) \geq (n-\ord(r))+r = n+(r-\ord(r))\geq n.
\end{align*}
Thus $\ord(\sum_{r=1}^{p^n} \binom{p^n}{r}x^{n-r}(py)^r)\geq n$ and the lemma follows.
\end{proof}

\begin{lemma}\label{lemma:tracepk}
Suppose $x\in \Cp$ is such that $\tau (x) =x^p$ and $x^{q^k}=x$, some $k\geq 1$. Then for any $b\in\mathbb{Z}$, $\Tr_{\mathbb{Q}_{q^k}/\Qp}(x^{p^b}) = \Tr_{\mathbb{Q}_{q^k}/\Qp}(x)$.
\end{lemma}
\begin{proof}
This is just a simple calculation: 
\begin{align*}
\Tr_{\mathbb{Q}_{q^k}/\Qp}(x^{p^b}) = \sum_{r=0}^{ak-1}\tau^r(x^{p^b}) = \sum_{r=0}^{ak-1}(x^{p^b})^{p^r} = \sum_{r=b}^{ak-1+b}x^{p^r} = \sum_{r=0}^{ak-1}x^{p^r} = \Tr_{\mathbb{Q}_{q^k}/\Qp}(x). 
\end{align*}
\end{proof}

We now make two observations. First, because $\widehat{x-P_j}$ is a Teichm\"uller lift, $\widehat{x-P_j}\equiv \widehat{x}-\widehat{P_j}\bmod p$ and so $\widehat{x-P_j}= (\widehat{x}-\widehat{P_j}) + pg(\widehat{x},\widehat{P}_j)$, some $g(X,Y)\in\mathbb{Z}[X,Y]$. Second, because $\widehat{x-P_j}$ is again a Teichm\"uller lift from $\mathbb{F}_{q^k}$, $\widehat{x-P_j} = (\widehat{x-P_j})^{q^k}=(\widehat{x-P_j})^{q^{bk}}$, any $b\geq 1$.  Thus, for any $b\geq 1$, applying Lemma~\ref{lemma:pkpower} yields:
\[
	\widehat{x-P_j} = (\widehat{x-P_j})^{q^{bk}} =((\widehat{x}-\widehat{P_j}) + pg(\widehat{x},\widehat{P}_j))^{q^{bk}} = (\widehat{x}-\widehat{P}_j)^{q^{bk}} + q^{bk}h(\widehat{x}, \widehat{P}_j).
\] 
If we take $b$ sufficiently large so that $q^{b}\geq p^{m-1}$, then 
\[	
	\widehat{x-P_j}\equiv (\widehat{x}-\widehat{P}_j)^{q^{bk}} \bmod p^m,	
\]
and so by Lemma~\ref{lemma:tracepk}, if $a\in \Fq$ then
\begin{align*}
	\zeta_{p^m}^{\Tr_{\Qqk/\Qp} (\widehat{a}^{p^{-i}}{(\widehat{x-P_j})}^{-p^{-i}})} &= \zeta_{p^m}^{\Tr_{\Qqk/\Qp} (\widehat{a}^{p^{-i}}(\widehat{x}-\widehat{P}_j)^{-p^{akb-i}}) + p^mh} \\&= \zeta_{p^m}^{\Tr_{\Qqk/\Qp} (\widehat{a}^{p^{-i}}(\widehat{x}-\widehat{P}_j)^{-p^{akb-i}})}\cdot \zeta_{p^m}^{p^mh} \\
	&=\zeta_{p^m}^{\Tr_{\Qqk/\Qp} (\widehat{a}(\widehat{x}-\widehat{P}_j)^{-1})}. 
\end{align*}

The above discussion induces a $p$-adic lifting of $f$ that preserves the exponential sum:
\begin{lemma}\label{f_lift} If 
$$\hat{f}=\sum_{i=0}^{m-1}\sum_{j=1}^\ell\sum_{k=0}^{d_{ij}}p^i\hat{a}_{ijk}\frac{1}{(x-\widehat{P}_j)^{k}}\in \Zq[\frac{1}{x-\widehat{P}_1}, \cdots, \frac{1}{x-\widehat{P}_\ell}],$$
then:
$$S_f(k) = \sum_{\substack{x\in\widehat{\mathbb{F}_{q^k}^\times},\\ x\neq\widehat{P}_1, \cdots, \widehat{P}_\ell}}\zeta_{p^m}^{\Tr_{\mathbb{Q}_{q^k}/\Qp}(\hat{f}(x))}.$$
\end{lemma}
\begin{proof}

If $x\in{\mathbb{F}_{q^k}^\times}$, clearly $\omega(f(x))=\hat{f}(\hat{x})$ and hence:
\begin{align*}
	\Tr_{W_m(\mathbb{F}_{q^k})/W_m(\Fp)}(f({x}_0)) &= \Tr_{\Qqk/\Qp} (\widehat{f}(\hat{x}_0)),
\end{align*}
and the lemma follows.
\end{proof}


\section{$p$-adic Banach Spaces}

Define a $p$-adic affinoid ring as follows.
\begin{definition}
For $0<r\leq 1$, let 
\[
	\mathbb{A}_r=\{x\in\Cp : |x|_p\leq 1/r, |x-\widehat{P_j}|_p\geq r\textrm{ for } 2\leq j\leq \ell\}.
\]	
It's easy to see that
\[
	\mathbb{A}_1=\{x\in\Cp : |x|_p=1\textrm{ and } |x-\widehat{P}_j|_p=1\textrm{ for } 2\leq j\leq \ell\}=\{x\in\widehat{\Fp^{alg}} : \bar{x}\neq P_j\}.
\]
\end{definition}
For convenience, we will often write $\mathbb{A}_1=\mathbb{A}$. 
\begin{definition}
Consider the affinoid space:
\begin{align*}
	\mathcal{H}^\dagger = \mathbb{A}\langle \frac{1}{x-\widehat{P}_1}, \cdots, \frac{1}{x-\widehat{P}_\ell}\rangle.
\end{align*}
	The ring $\mathcal{H}^\dagger$ is the set of overconvergent anaytic elements on $\mathbb{A}$. (That is, any $h(x)\in \mathbb{A}[[ \frac{1}{x-\widehat{P}_1}, \cdots, \frac{1}{x-\widehat{P}_\ell}]]$ lies in $\mathcal{H}^\dagger$ if and only if it can be evaluated at any $x\in\mathbb{A}$.) 
\end{definition}

In this section we'll study some fundamental properties of $\mathcal{H}^\dagger$ and look at some special quotient spaces. We start with a basic result, the Mittag-Leffler decomposition

\begin{prop}
If $\mathcal{H}^\dagger_j= \mathbb{A}\langle \frac{1}{x-\widehat{P}_j}\rangle$, there exists an isomorphism of $\mathbb{A}$-Banach modules
	$$\mathcal{H}^\dagger \cong \bigoplus_{j=1}^\ell \mathcal{H}^\dagger_j,$$
such that every $g$ in $\mathcal{H}^\dagger$ can be written $g=\sum_{j=1}^\infty (g)_j$, with $(g)_j\in\mathcal{H}^\dagger_j$.
\end{prop}
\begin{proof}
	See Lemma~2.1 in \cite{Zhu}. 
\end{proof}

\subsection{Quotient Spaces of $\mathcal{H}^\dagger$}\label{quotientspaces}

This section (along with the following Dwork cohomology) is based on the methods used by Lauder and Wan in their papers \cite{LauderWan1} and \cite{LauderWan2}.

For convenience in this subsection, we will write $X_j = \frac{1}{x-\widehat{P}_j}$. Fix an arbitrary polynomial $H = \sum_{j=1}^\ell\sum_{i=0}^{R_j}h_{ij}X_j^i$ in $\mathbb{A}[X_1,\cdots, X_\ell]$, with $R_j\geq 0$ and $h_{R_j, j}\neq 0$ and say $H_j = \sum_{i=0}^{R_j}h_{ij}X_j^i$.
Define two operators on $\mathcal{H}^\dagger$:
\begin{align*}
	D &= EH + E\\
	D_j &= EH_j+E,
\end{align*}
where $E = x\frac{d}{dx}$.

We wish to understand the quotient space $\mathcal{H}^\dagger/D(\mathcal{H}^\dagger)$. Namely, we are interested in its dimension as a module over $\mathbb{A}$. Our goal will be to show that $\mathcal{H}^\dagger/D(\mathcal{H}^\dagger)$ is isomorphic to the following finite free $\mathbb{A}$-module $R$:
\begin{definition}
Consider the set
\begin{align*}
B = \{1, X_1, \cdots, X_1^{R_1-1}, X_2, \cdots, X_2^{R_2}\}\cup \{ X_j^i | j\geq 3, 1\leq i\leq R_j+1\}
\end{align*}
and define the $\mathbb{A}$-module $R =\vspan_{\mathbb{A}} B\subset \mathcal{H}^\dagger$.
\end{definition}

For the next lemma we will need some notational sugar:
\begin{align*}
R_j' &= \begin{cases}
	R_j-1 &\textrm{ if } j=2\\
	R_j &\textrm{ if } j=1\\
	R_j+1 &\textrm{ if } j\geq 3
\end{cases}
\end{align*}

\begin{lemma}\label{r_cong}
Fix $1\leq j \leq \ell$ and take $u\geq R_j'$.  Let 
$$r_{j,u}=\frac{1}{R_jh_{R_j,j}(-\widehat{P_j})^{\delta(j)}}X_j^{u-R_j'}\in\mathcal{H}^\dagger_j,$$
where $\delta(j)=1$ if $j\geq 3$ and $\delta(j)=0$ if $j=1,2$. (Take $(\infty)^0=0^0=1$.) There is then a congruence
	$$D_jr_{j,u} = (EH_j+E)r_{j,u} \equiv X_j^u\bmod R.$$
\end{lemma}
\begin{proof}
Observe that the action of $E$ on the terms $X_j^i$ is nothing but:
\begin{align*}
E(X_j^i) = 
\begin{cases}
	iX_j^i &\textrm{ if } j=1, 2\\
	-iX_j^i-i\widehat{P_j}X_j^{i+1} &\textrm{ if } j\geq 3.
\end{cases}
\end{align*}
So if $j=1$:
\begin{align*}
(EH_1+E)r_{1,u} &= (\sum_{i=0}^{R_1}ih_{i,1}X_1^i+E)\frac{1}{R_1h_{R_1,1}}X_1^{u-R_1}\\
	&= \sum_{i=0}^{R_1}\frac{ih_{i,1}}{R_1h_{R_1,1}}X_1^{u-R_1+i}+E\circ \frac{1}{R_1h_{R_1,1}}X_1^{u-R_1}\\
	&= X_1^u + \sum_{i=0}^{R_1-1}\frac{ih_{i,1}}{R_1h_{R_1,1}}X_1^{u-R_1+i}+\frac{u-R_1}{R_1h_{R_1,1}}X_1^{u-R_1}.
\end{align*}
All the terms except $X_1^u$ have degree less than or equal to $u-1$. When $u=R_1'=R_1$, then 
	$(EH_1+E)r_{1,R_1} \equiv X_1^{R_1}\bmod R$,
and so if $u>R_1$, by induction $(EH_1+E)r_{1,u} \equiv X_1^{u}\bmod R$. The proof is similar for $j=2$.

For $j\geq 3$, first compute that 
\begin{align*}
EH_j &= \sum_{i=0}^{R_j}h_{ij}(-iX_j^i-i\widehat{P_j}X_j^{i+1}) = -h_{R_j, j}R_j\widehat{P_j}X^{R_j+1} + \sum_{i=1}^{R_j}h_{ij}'X_j^i,
\end{align*}
some $h_{ij}'$. Hence:
\begin{align}\label{Ejgreater3}
\nonumber(EH_j+E)r_{j,u} &= (-h_{R_j, j}R_j\widehat{P_j}X^{R_j+1} + \sum_{i=1}^{R_j}h_{ij}'X_j^i+E)\frac{1}{R_jh_{R_j,j}(-\widehat{P_j})}X_j^{u-(R_j+1)}\\
	&= X_j^{u} + \sum_{i=1}^{R_j}\frac{h_{ij}'}{R_jh_{R_j,j}(-\widehat{P_j})}X_j^{u-(R_j+1)+i}+E\circ\frac{1}{R_jh_{R_j,j}}X_j^{u-(R_j+1)}.
\end{align}
Because $$E\circ X_j^{u-(R_j+1)}=-(u-(R_j+1))X_j^{(u-(R_j+1))}-(u-(R_j+1))\widehat{P_j}X_j^{u-(R_j+1)+1},$$
the terms in (\ref{Ejgreater3}) except for $X_j^u$ have degree strictly less than $u$, and so by the same induction argument with base case $u=R_j+1$, the claim follows.
\end{proof}

For any $g(x)\in\mathcal{H}_j^\dagger/D(\mathcal{H}_j^\dagger)$, Lemma~\ref{r_cong} implies that we can write uniquely 
	$$g(x) \equiv \sum_{i=0}^{R_j'} a_{ij}X_j^i\bmod D(\mathcal{H}_j^\dagger),$$
and so it now remains to decompose $\mathcal{H}^\dagger/D(\mathcal{H}^\dagger)$ in terms of $\mathcal{H}_j^\dagger/D(\mathcal{H}_j^\dagger)$:
\begin{theorem}\label{H_decomp}
$\dim_A \mathcal{H}^\dagger/D(\mathcal{H}^\dagger)= \sum_{j=1}^\ell \dim_A\mathcal{H}_j^\dagger/D(\mathcal{H}_j^\dagger)$.
\end{theorem}
\begin{proof}

For a fixed $j_0$ and $i_0$, we have by partial fraction expansion
\begin{align*}
\sum_{j=1, j\neq j_0}^\ell EH_jX_{j_0}^{i_0} &= \sum_{j=1, j\neq j_0}^\ell \sum_{i=0}^{R_j'}h_{ij}'X_j^i + \sum_{i=0}^{i_0}h_{ij_0}'X_{j_0}^i\\
	&\equiv  \sum_{i=0}^{i_0}h_{ij_0}'X_{j_0}^i \bmod R.
\end{align*}

Hence for appropriate $j$ and $u$, by induction on $u$ (with base case $u=R_j'$), it follows that 
\begin{align*}
	(EH+H)r_{j,u} &\equiv (EH_j+E)r_{j,u} + \sum_{i=0}^{u-R_j'}h_{ij}'X_j^i
		\equiv (EH_j+E)r_{j,u}\\&\equiv X_j^u\bmod R,
\end{align*}
and so the Theorem follows because every $g(x)\in\mathcal{H}^\dagger/D(\mathcal{H}^\dagger)$ can be written uniquely as 
	$$g(x) \equiv \sum_{j=1}^\ell\sum_{i=0}^{R_j'} a_{ij}X_j^i\bmod D(\mathcal{H}^\dagger).$$
\end{proof}


\section{The Degree of the $L$-function}\label{Dwork_coh}

We will now apply the prior section's results about the space $\mathcal{H}^\dagger$ to construct our Dwork theory. 

\subsection{The truncated Artin-Hasse Exponential}

\begin{definition} Let $1\leq k\leq m$. Define the $k$-truncated Artin-Hasse exponential:
$$E_k(x)=\exp(\sum_{i=0}^k\frac{x^{p^i}}{p^i})=\sum_{i=0}^\infty u_{ki}x^i,$$
and take $\pi_k\in\Cp$ to be a solution to $\sum_{i=0}^k\frac{x^{p^i}}{p^i}=0$ with $\ord \pi_k = \frac{1}{p^{k-1}(p-1)}$.
\end{definition}

By Theorem 4.1 in \cite{Conrad}, the disk of convergence of $E_k(x)$ is
	\[
		D_{k+1} = \{ x : |x|_p < p^{-(k+1+\frac{1}{p-1})/p^{k+1}}\}.
	\] 
Furthermore, for any $j\leq k$, by Theorem 4.9 also in \cite{Conrad}, because $p>m\geq k$, it is guaranteed that $\pi_k$ lies inside the disc of convergence of $E_k$ and it's therefore well defined to consider the splitting functions 
	\[
		\theta_k(x)=E_k(\pi_k x)\textrm{ and } \hat{\theta}_k=\prod_{j=0}^\infty \theta_k(x^{p^j}).
	\]	 

\begin{lemma}\label{theta_roc}
If $\theta_k(x)=\sum_{i=0}^\infty \theta_{ki}x^i$, then $\ord\theta_{ki}\geq i\cdot \frac{p-k}{p^{k+1}}$.
\end{lemma}
\begin{proof}
By Remark 4.5 in \cite{Conrad}, $\ord u_{ki}\geq -\frac{i}{p^{k+1}}(\frac{1}{p-1}+k+1)$, and so 
\begin{align*}
\ord \theta_{ki}\geq -\frac{i}{p^{k+1}}(\frac{1}{p-1}+k+1)+\frac{i}{p^{k-1}(p-1)} = i\cdot\frac{p-k}{p^{k+1}}.
\end{align*}
\end{proof}

\begin{lemma} For $k\geq 1$ and  $0\leq j\leq k-1$ let $\gamma_{kj}=\sum_{i=0}^j\frac{\pi_k^{p^i}}{p^i}$. Then:
$$\hat{\theta}_k(x)=\exp(\sum_{j=0}^{k-1}\gamma_{kj}x^{p^j}),$$
and $\ordp \gamma_{kj} = \frac{1}{p^{k-1-j}(p-1)}-j$.
\end{lemma}
\begin{proof}
A simple computation shows that:
\begin{align*}
 \sum_{v=0}^\infty\sum_{i=0}^k \frac{\pi_k^{p^i}}{p^i}x^{p^{i+v}}&=\sum_{v=k}^\infty\left (\sum_{i=0}^k\frac{\pi_k^{p^i}}{p^i}\right )x^{p^v}+\sum_{v=0}^{k-1}\left (\sum_{\substack{v=i_1+i_2\\0\leq i_1\leq k-1\\0\leq i_2\leq k-1}}\frac{\pi_k^{p^{i_1}}}{p^{i_1}}\right )x^{p^v}\\
	&= \sum_{v=0}^{k-1}\left (\sum_{i=0}^v \frac{\pi_k^{p^i}}{p^i}\right )x^{p^v},
\end{align*}
and so the identity follows by the defining property of $\pi_k$. To compute the order of $\gamma_{kj}$, observe that 
	$$\ord (\sum_{0\leq i\leq j}\frac{\pi_k^{p^i}}{p^i})\geq \min_{0\leq i\leq j}(\frac{1}{p^{k-1-i}(p-1)}-i),$$
	which has a unique minimum at $i=j$. 
\end{proof}

\begin{lemma}\label{dwork_trace_comm}
	Let $k,v\geq 1$, $t\in\Cp$ and suppose $x\in\Cp$ such that $x^{p^v}=x$. Then 
	\begin{align*}
			E_k(t)^{\sum_{i=0}^{v-1}x^{p^i}} &= \prod_{i=0}^{v-1} E_k(tx^{p^i}).
	\end{align*}	
\end{lemma}
\begin{proof}The lemma follows from the following computation:
	\begin{align*}
			E_k(t)^{\sum_{i=0}^{k-1}x^{p^i}} &= \exp(\sum_{j=0}^k\frac{t^{p^j}}{p^j}\cdot \sum_{i=0}^{k-1}x^{p^i}) = \exp(\sum_{j=0}^k\frac{t^{p^j}}{p^j} \sum_{i=0}^{v-1}x^{p^{i+j}}) \\
				&=\exp(\sum_{j=0}^k \sum_{i=0}^{v-1}\frac{(x^{p^i}t)^{p^j}}{p^j}).
	\end{align*}	
\end{proof}

\subsection{Dwork Theory via the Truncated Artin-Hasse}

\begin{definition}\label{EHEEdef}
Define:
\begin{align*}
	E_f(x)&=\prod_{i=0}^{m-1}\prod_{j=1}^\ell\prod_{k=0}^{d_{ij}}{\theta_{m-i}}(\hat{a}_{ijk}\frac{1}{(x-{P}_j)^{k}})\\
	E_f^{(a)}&=\prod_{j=0}^{a-1}E_f^{\tau_j}(x^{p^j}).
\end{align*}
\end{definition}

The reason for these definitions is clear. For $x\in \widehat{\mathbb{F}_{q^w}}$, $w\geq 1$, applying the splitting function $\theta$ to $\hat{f}$ yields:
\begin{align*}
	\zeta_{p^m}^{\Tr_{\mathbb{Q}_{q^k}/\Qq}(\hat{f}({x}))}  &= \zeta_{p^m}^{{\sum_{i=0}^{m-1}\sum_{j=1}^\ell\sum_{k=0}^{d_{ij}}p^i\Tr_{\mathbb{Q}_{q^k}/\Qq}(\hat{a}_{ijk}\frac{1}{(x-\widehat{P}_j)^{k}}})}\\
	&= \prod_{i=0}^{m-1}\prod_{j=1}^\ell\prod_{k=0}^{d_{ij}}\zeta_{p^{m-i}}^{{\Tr_{\mathbb{Q}_{q^k}/\Qq}(\hat{a}_{ijk}\frac{1}{(x-\widehat{P}_j)^{k}}})} \\
	&= \prod_{i=0}^{m-1}\prod_{j=1}^\ell\prod_{k=0}^{d_{ij}}E_{m-i}(\pi_{m-i})^{\Tr_{\mathbb{Q}_{q^k}/\Qq}(\hat{a}_{ijk}\frac{1}{(x-\widehat{P}_j)^{k}})}\\
	&= \prod_{i=0}^{m-1}\prod_{j=1}^\ell\prod_{k=0}^{d_{ij}}\prod_{v=0}^{aw-1}E_{m-i}(\pi_{m-i}\frac{\hat{a}_{ijk}^{p^v}}{(x-\widehat{P}_j)^{kp^v}}) = \prod_{v=0}^{aw-1}E_f^{\tau^v}(x^{p^v}).
\end{align*}
\begin{remark}
	Note that when $x$ is a Teichm\"uller lift, $(x-\widehat{P}_j)^p \equiv x^p-\widehat{P}_j^p\bmod p$ and using the same argument prior to Lemma~\ref{f_lift} we see that the last equality above really does hold.	
\end{remark}

Similar to the above, define the analogous functions: 	
	$$\hat{E}_f=\prod_{j=0}^\infty E_f^{\tau^j}(x^{p^j})=\prod_{i=0}^{m-1}\prod_{j=1}^\ell\prod_{k=0}^{d_{ij}}\hat{\theta}_{m-i}(\hat{a}_{ijk}\frac{1}{(x-\widehat{P}_j)^{k}}) = \exp(H),$$
so that
	$$ H = \log\hat{E}_f =  \sum_{i=0}^{m-1}\sum_{v=0}^{m-i-1}\sum_{j=1}^\ell\sum_{k=0}^{d_{ij}}\gamma_{m-i,v} \hat{a}_{ijk}^{p^v}\frac{1}{(x-\widehat{P}_j)^{p^{v}k}}.$$
	
\begin{remark}
	For each $j$ and $i$, the coefficient of $\frac{1}{(x-\widehat{P}_j)^{p^{m-i-1}d_{ij}}}$ is equal to \sloppy$\gamma_{m-i, m-i-1}a_{ij,d_{ij}}$, and since $a_{ij,d_{ij}}\neq 0$ by assumption and $\gamma_{m-i, m-i-1}=\sum_{i=0}^{m-1-1}\frac{\pi_{m-i}^{p^i}}{p^i}=-\frac{\pi_{m-i}^{m-i}}{p^{m-i}}\neq 0$, the entire coefficient is nonzero. Thus $\deg_j H = \max_i p^{m-i-1}d_{ij}$.  
\end{remark}
	
From these functions define the corresponding Dwork maps $\alpha_{1} = U\circ E_f$ and $\alpha_{a} = U^a\circ E_f^{(a)}$.

\begin{prop}\label{alpha_cc}
The maps $\alpha_1$ and $\alpha_a$ are $p$-adically completely continuous. 
\end{prop}
\begin{proof}
	This proof is analogous to the corresponding proof in Corollary~6.10 in \cite{Schmidt1}. Fix $j$, let $i_j$ be such that $d_{i_j,j} = \max_i d_{ij}$ and define
	\begin{align*}
		F_j(x) &= \prod_{i=0}^{m-1}\prod_{k=1}^{d_{ij}} \theta_{m-i}({\frac{\widehat{a_{ijk}}}{(x-\widehat{P}_j)^k}}) =  \prod_{i=0}^{m-1}\prod_{k=1}^{d_{ij}} \left ( \sum_{v=0}^\infty \theta_{m-i, v}\widehat{\frac{a_{ijk}}{(x-P_j)^k}}^v\right )\\
		&= \sum_{n=0}^\infty \left	 ( \sum_{\substack{\sum_{v=1}^{d_{i_0,j}} v n_{i_v,v}=n\\0\leq i_v\leq m-1,\\n_{i_v,v}\geq	 0}} \prod_{v=1}^{d_{i_0,j}} a_{i_v,j,v}^{n_{i_v,v}}\theta_{m-i_v,n_{i_v,v}}\right ) \frac{1}{(x-\widehat{P}_j)^n}.
\end{align*} 
In this case, $\ord(F_{nj})\geq \frac{n}{d_{i_j,j}}\cdot\frac{p-(m-i_j)}{p^{m-i_j+1}}$ by Lemma~\ref{theta_roc}. (Note that we are using here that $\theta_{ki}=u_{ki}\pi_{k}^i\neq 0$ because $u_{kn}=\frac{\Hom(C_{p^k}, S_n)}{n!}$ and hence $u_{kn}\neq 0$ ) The main observation to make is that all of the proofs following and including Lemma~6.5 in \cite{Schmidt1} can be generalized to the assumption that  $\ord(F_{nj})\geq {\frac{n}{d_{i_j,j}}}\cdot M_j$, where $M_j$ is any real number with $0<M_j<1$. Here, we see that $M_j=\frac{p-(m-i_j)}{p^{m-i_j+1}}$ and the rest of the theory  follows accordingly.  
	
\end{proof}

Recall from subsection~\ref{quotientspaces} that we defined the two operators $E = x\frac{d}{dx}$ and $D = E + E\circ\log \hat{E}_f$.

\begin{lemma}\label{hom_relations}
We have the following relations:
\begin{enumerate}
\item $U^a\circ E = E\circ (qU^a)$. 
\item $D =  \exp(-H)\circ E\circ \exp(H)$.
\item $\alpha_a = \exp(-H)\circ U^a\circ \exp(H)$.
\item $D\circ (q\alpha_a)=\alpha_a\circ D$.
\end{enumerate}
\end{lemma}
\begin{proof}
Let $g\in \mathcal{H}^\dagger$.
\begin{enumerate}
\item Note that if $z^q=x$, then $\frac{dz}{dx}=\frac{1}{q}x^{\frac{1}{q}-1}$  and so $x\frac{dz}{dx}=\frac{1}{q}x^{\frac{1}{q}}=\frac{1}{q}z$. Hence,
	\begin{align*}
		E\circ(qU^ag) &= x\frac{d}{dx}\sum_{z^q=x}g(z)=x\sum_{z^q=x}\frac{dg}{dx}(z)\cdot \frac{dz}{dx}\\&=\frac{1}{q}\sum_{z^q=x}z\frac{dg}{dx}(z)=U^a\circ Eg.
	\end{align*}
\item The identity is just the simple calculation:
	\begin{align*}
		\exp(-H)\circ x\frac{d}{dx}\circ (\exp(H)g) =& \exp(-H)\circ (x(\exp(H)\frac{dH}{dx}g + \exp(H)\frac{dg}{dx}))\\
		&= \exp(-H)\exp(H)x(\frac{dH}{dx}g+\frac{d}{dx}g)\\
		&=(x\frac{dH}{dx}+x\frac{d}{dx})g = D.
	\end{align*}
\item The equality is equivalent to the statement $\alpha_a = U^a(\hat{E}_f)/\hat{E}_f$. By basic properties of $U$ and observing that $\hat{E}_f = \prod_{j=0}^\infty E_f^{\tau^j}(x^{p^j})$, 
	\begin{align*}
		U^a(\prod_{j=0}^\infty E_f^{\tau^j}(x^{p^j}))&=U^a(\prod_{j=0}^{a-1} E_f^{\tau^j}(x^{p^j})\prod_{j=a}^\infty E_f^{\tau^j}(x^{p^j}))\\&=\prod_{j=a}^\infty E_f^{\tau^j}(x^{p^{j-a}})\cdot U^a(E_f^{(a)})
	=\prod_{j=a}^\infty E_f^{\tau^{j+a}}(x^{p^j})\cdot U^a(E_f^{(a)})\\
	&=\hat{E}_f(x)\cdot U^a(E_f^{(a)}) = \hat{E}_f\cdot \alpha_a.
	\end{align*}
\item  This identity follows from the first three. 
\end{enumerate}
\end{proof}

Lemma~\ref{hom_relations} yields the commutative diagram:
\begin{equation}\label{comm_diag}
\begin{tikzcd}
0 \arrow[r] & \ker D \arrow[r] \arrow[d, "q\alpha_a"] &  \mathcal{H}^\dagger \arrow[d, "q\alpha_a"] \arrow[r, "D"] & \mathcal{H}^\dagger  \arrow[r] \arrow[d, "\alpha_a"] & H_0 \arrow[r] \arrow[d, "\bar{\alpha}_a"] & 0\\
0 \arrow[r] & \ker D \arrow[r] & \mathcal{H}^\dagger \arrow[r, "D"] & \mathcal{H}^\dagger \arrow[r] & H_0 \arrow[r] &0 
\end{tikzcd}
\end{equation}
where $H_0=\mathcal{H}^\dagger/D(\mathcal{H}^\dagger)$.

First, we prove that $D$ is injective.
\begin{lemma}$D$ is injective.
\end{lemma}
\begin{proof}
Suppose that $g\in \ker D$ so that 
\begin{align*}
	Dg &= \exp(-H)\cdot E(\exp(H)g)=\exp(-H)\cdot x(\exp(H)\frac{dH}{dx}g+\exp(H)\frac{dg}{dx})\\
          	&= x(g\frac{dH}{dx}+\frac{dg}{dx})=0,
\end{align*}
which implies that $g\frac{dH}{dx}+\frac{dg}{dx}=0$, or 
\begin{align*}
-\frac{dH}{dx}=\frac{\frac{dg}{dx}}{g}\rightarrow -H+C=\ln(g)\rightarrow g=c\exp(-H),
\end{align*}
some scalar $c$. It remains to show that $\exp(-H)\not\in\mathcal{H}$. 

Recall that $\exp(x)$ converges on the disk $\{x\in\Cp | \ord x > 1/(p-1)\}$. So let $H(x) = \sum_{v,i,j,k}h_{vijk}\frac{1}{(x-\widehat{P_j})^{p^{v}k}}$ and
\begin{align*}
 \ord h_{vijk} = \ord\gamma_{m-i,v}= \frac{1}{p^{m-i-1-v}(p-1)}-v.
\end{align*}
Hence 
	$$\min_{vijk} \ord h_{ijvk} = \min_{vijk} (\frac{1}{p^{m-i-1-v}(p-1)}-v) = \frac{1}{p-1}-(m-1), $$
	taking $i=0$ and $v=m-i-1$. But this implies that $|H|_{gauss} = p^{-(\frac{1}{p-1}-(m-1))} = p^{m-1}p^{-1/(p-1)} \geq p^{-1/(p-1)}$.
	
	By a well known property of the gauss norm, we can therefore find $x_0\in\mathbb{A}_1$ such that $|H(x_0)|_p=p^{m-1}p^{-1/(p-1)}$, and so at this $x_0$, $\exp(H(x_0))$ is not defined, and $\exp(-H)\not\in\mathcal{H}^\dagger$.
\end{proof}

\begin{prop}
The degree of the $L$-function is the degree of the first homology space, $\deg L_f(s)=\dim_{\Zq[[\pi_m]]} H_0$.
\end{prop}
\begin{proof}
Just as in Theorem~6.13 in \cite{Schmidt1}, because $\alpha_a$ is completely continuous by Proposition~\ref{alpha_cc}, we see that $C_f(s)=  \det(1-\alpha_as)$. A simple computation then shows that 
	$$L_f(s) = \frac{C_f(s)}{C_f(qs)}=\frac{\det(1-\alpha_as)}{\det(1-q\alpha_as)}.$$
Since $\mathcal{H}^\dagger = D(\mathcal{H}^\dagger)\oplus \mathcal{H}^\dagger/D(\mathcal{H}^\dagger)$, 
\begin{align}\label{trace_starting}
	\Tr(\alpha_a | \mathcal{H}^\dagger) = \Tr(\alpha_a\big |_{D(\mathcal{H}^\dagger)})+\Tr(\alpha_a\big |_{\mathcal{H}^\dagger/D(\mathcal{H}^\dagger)}).
\end{align}
But by the relation $\alpha_a\circ D = D\circ (q\alpha_a)$ from Lemma~\ref{hom_relations}, $D^{-1}\circ \alpha_a\circ D = q\alpha_a$ (since $D$ is bijective on $D(\mathcal{H}^\dagger)$). Thus:
	$$\Tr(\alpha_a\big |_{D(\mathcal{H}^\dagger)}) = \Tr(D^{-1}\circ \alpha_a\big |_{D(\mathcal{H}^\dagger)}\circ D) = \Tr(q\alpha_a | \mathcal{H}^\dagger).$$

And so (\ref{trace_starting}) becomes:
	$$\Tr(\alpha_a | \mathcal{H}^\dagger) = \Tr(q\alpha_a | \mathcal{H}^\dagger) + \Tr(\bar{\alpha}_a | H_0),$$
that is $\Tr(\alpha_a | \mathcal{H}^\dagger)-\Tr(q\alpha_a | \mathcal{H}^\dagger) = \Tr(\bar{\alpha}_a| H_0),$ and the same formula holds when $\alpha_a$ is replaced with $\alpha_a^k$ since $D\circ (q^k\alpha_a^k) = \alpha_a^k\circ D$. Thus by the well known formula
	$$\det(1-\phi s) = \exp(-\sum_{k=1}^\infty\Tr\phi^k\frac{s^k}{k}),$$
we have the identity
	$$L_f(s) = \det(1-\bar{\alpha}_as | H_0),$$
and the claim follows.
\end{proof}

\begin{theorem}
$\dim_{\Zq[[\pi_m]]} H_0= (\sum_{j=1}^\ell (\max_{0\leq i\leq m-1} p^{m-i-1}d_{ij}+1))-2$.
\end{theorem}
\begin{proof}
We apply the results of subsection~\ref{quotientspaces} to the polynomial $H$ in definition~\ref{EHEEdef}. By Lemma~\ref{r_cong}, it's clear that:
\begin{align*}
\dim_{\mathbb{A}} \mathcal{H}^\dagger_j/D_j\mathcal{H}^\dagger_j =
\begin{cases}
	R_j -1 &\textrm{ if } j=1\\
	R_j+1 &\textrm{ if } j\geq 2,
\end{cases}
\end{align*}
where $R_j = \max_{0\leq i\leq m-1} p^{m-i-1}d_{ij}$. With this computation, the degree of $L_f$ follows  from Theorem~\ref{H_decomp}.
\end{proof}

\section{The Hodge Polygon}

In this second section we utilize a more classical style Dwork theory from the Artin-Hasse exponential and compute the Hodge polygon of the $L$-function and generalize the main result of \cite{Zhu}. The theory in this section is analogous to the work from \cite{Schmidt1}, but we provide a detailed sketch with a focus on $p$-adic estimates that we will use to compute the Newton polygon. 

\subsection{The Artin-Hasse Exponential}

\begin{definition} Let $1\leq k\leq m$. Define the Artin-Hasse exponential:
\[
	E(x) = E_\infty(x)=\exp(\sum_{i=0}^\infty\frac{x^{p^i}}{p^i})=\sum_{i=0}^\infty u_{\infty,i}x^i\in\Zp[[x]],
\]
and take $\pi_{\infty,k}\in\Cp$ to be a solution to $\sum_{i=0}^\infty\frac{x^{p^i}}{p^i}=0$ with $\ord \pi_{\infty,k} = \frac{1}{p^{k-1}(p-1)}$.
\end{definition}

Define the splitting function:
\[
	\theta_{\infty, k}(x) = E(\pi_{\infty,k} x) = \sum_{i=0}^\infty u_{\infty,i}x^i
\]
When there is no confusion, we will write $u_i= u_{\infty, i}$ and $\pi_{\infty,i}=\pi_i$. 

\subsection{$p$-adic Estimates}

\begin{definition}
	Let $i\geq 0$, $1\leq j\leq \ell$:
	
	\begin{align*}
	F_{ij} &= \prod_{k=0}^{d_{ij}} E(\pi_{m-i}\widehat{a_{ijk}}\frac{1}{(x-\widehat{P_j})^k}) = \sum_{n=0}^\infty F_{ij,n}\frac{1}{(1-\widehat{P_j})^n}\\
	F_j &= \prod_{i=0}^{m-1} F_{ij} = \sum_{n=0}^\infty F_{j,n}\frac{1}{(x-\widehat{P_j})^n}\\
	F &=\prod_{j=1}^\ell F_{j}\\
	(F\frac{1}{(x-\widehat{P_j})^i})_k &= \sum_{n=0}^\infty F_{ij,nk}\frac{1}{(x-\widehat{P_k})^n}\\
	\alpha_1 &= U\circ F\\
	\alpha_a &= \alpha_1^a\\
	(\alpha_1\frac{1}{(x-\widehat{P_j})^i})_k &= \sum_{n=0}^\infty C_{ij,nk}\frac{1}{(x-\widehat{P_k})^n}
	\end{align*}
\end{definition}

\begin{lemma}\label{lemma_Fijn}
For $i,n\geq 0$ and $1\leq j\leq \ell$, 
		$$\ord F_{ij,n} \geq \frac{n}{d_{ij}p^{m-i-1}(p-1)},$$
with equality if and only if $d_{ij}|n$ and $u_{\frac{n}{d_{ij}}}\in\Zptimes$ (which is satisfied if $\frac{n}{d_{ij}}<p$).
\end{lemma}
\begin{proof}
We compute that:
		\begin{align*}
		F_{ij}(x) &= \prod_{k=1}^{d_{ij}} E(\pi_{m-i}\widehat{a_{ijk}}\frac{1}{(x-\widehat{P_j})^k})
			      =  \prod_{k=1}^{d_{ij}} \left ( \sum_{v=0}^\infty u_v \widehat{a_{ijk}}^v\pi_{m-i}^v\frac{1}{(x-\widehat{P_j})^{kv}} \right )\\
		&= \sum_{n=0}^\infty \left ( \sum_{\substack{\sum_{k=1}^{d_{ij}} kn_k=n\\n_k\geq 0}} \prod_{k=0}^{d_{ij}} u_{n_k} a_{ijk}^{n_k}\pi_{m-i}^{n_k}\right ) \frac{1}{(x-\widehat{P_j})^n}.
\end{align*} 
Hence,
\begin{align}\label{Fijn}
	F_{ijn} = \sum_{\substack{\sum_{k=1}^{d_{ij}} kn_k=n\\n_k\geq 0}} \prod_{k=0}^{d_{ij}} u_{n_k} a_{ijk}^{n_k}\pi_{m-i}^{n_k},
\end{align}
and so the lemma is clear.
\end{proof}

\begin{lemma}
For $i,n\geq 0$ and $1\leq j\leq \ell$, 
	\[
		\ord F_{j,n} \geq \frac{n}{d_{i_jj}p^{m-i_j-1}(p-1)},
	\]
with equality if and only if $d_{i_jj}|n$ and $u_{\frac{n}{d_{i_jj}}}\in\Zptimes$ (which is satisfied if $\frac{n}{d_{i_jj}}<p$).
\end{lemma}
\begin{proof} Observe that
	\begin{align*}
		F_j &= \prod_{i=0}^{m-1}\left ( \sum_{k=0}^\infty F_{ij,k}\frac{1}{(x-\widehat{P_j})^k}\right ) \\
			&= \left ( \sum_{\substack{n_0,\cdots, n_{m-i}\in\mathbb{Z}_{\geq 0}\\ n_0+\cdots + n_{m-1} = n}} \prod_{k=0}^{m-1} F_{kj,n_k}\right ) \frac{1}{(x-\widehat{P_j})^n}.
	\end{align*}
Then by Lemma~\ref{lemma_Fijn},
\begin{align*}
	\ord  \prod_{k=0}^{m-1} F_{kj,n_k} &= \sum_{k=0}^{m-1}\frac{n_k}{d_{kj}p^{m-k-1}(p-1)}\geq \sum_{k=0}^{m-1}\frac{n_k}{d_{i_jj}p^{m-i_j-1}(p-1)}\\&=\frac{n}{d_{i_jj}p^{m-i_j-1}(p-1)},
\end{align*}
and this minimum is achieved when $n_{i_j}=n$ and $n_k=0$ otherwise. Furthermore, this minimum is unique, under our assumption that $i_j$ is the unique maximum in $d_{i_jj}p^{m-i_j-1}=\max_id_{ij}p^{m-i-1}$ and obtained if and only if  $d_{i_jj}|n$ and $u_{\frac{n}{d_{i_jj}}}\in\Zptimes$.
\end{proof}

\begin{lemma}\label{lemma:Fijnk}
For $1\leq j, k\leq \ell$ and $0\leq i,n$:
\[
	\ord F_{ij,nk} \geq \frac{n-i}{d_{i_k,k}p^{m-i_k-1}(p-1)}
\]
and equality holds if and only if $j=k$, $d_{i_k,k}|(n-i)$ and $u_{\frac{n-i}{d_{i_k,k}}}\in\Zptimes$.
\end{lemma}
\begin{proof}
We follow Lemma 4.4.7 in \cite{Schmidt1}. First note that:
\[
	\ord F_{ij,nk} \geq \min \prod_{v=1}^\ell F_{v, n_v},
\]
where $(n_1,\cdots, n_\ell)\in\mathbb{Z}_{\geq 0}^\ell$ such that $n_k-\sum_{\substack{v=1\\v\neq k}} n_v = n-i$. 
\end{proof}

\begin{prop}
	For $1\leq j, k\leq \ell$ and $0\leq i,n$:
\[
	\ord C_{ij,nk} \geq \frac{(n-1)p-i+1}{d_{i_k,k}p^{m-i_k-1}(p-1)},
\]
with equality if and only if $j=k$, $d_{i_k,k}|((n-1)p-i+1)$ and $u_{\frac{(n-1)p-i+1}{d_{i_k,k}}}\in\Zptimes$.
\end{prop}
\begin{proof}
Following Proposition 4.4.9 in \cite{Schmidt1} and using Lemma~\ref{lemma:Fijnk},
\begin{align*}
	\ord C_{ij,nk} &\geq \min_{n\leq v\leq (n-1)p+1}\left (\frac{v-i}{d_{i_kk}p^{m-i_k-1}(p-1)}+\frac{np-v}{p-1}-1\right )\\
		&= \min_{n\leq v\leq (n-1)p+1}\left (\left(\frac{1}{D_k(p-1)}-\frac{1}{p-1}\right )v + \left (\frac{D_knp-i}{D_k(p-1)}-1\right )\right ),
\end{align*}
where we write $D_k = d_{i_kk}p^{m-i_k-1}$ to ease notation. The minimum is then uniquely achieved when $v=(n-1)p+1$ since $\frac{1}{D_k(p-1)}-\frac{1}{p-1}<0$. 
\end{proof}

\begin{corollary}\label{corollary:D_ijnk}
	For $1\leq j, k\leq \ell$ and $0\leq i,n$:
\[
	\ord D_{ij,nk} \geq \frac{n-1}{d_{i_k,k}p^{m-i_k-1}},
\]
with equality if and only if $j=k$, $d_{i_k,k}|((n-1)(p-1))$ and $u_{\frac{(n-1)(p-1)}{d_{i_k,k}}}\in\Zptimes$.
\end{corollary}

\subsection{Computing the Hodge}

Let $M$ represent the matrix for $\alpha_1$ with respect to the weighted basis $\{\pi_{m-i_j}^{i/d_{i_jj} }\frac{1}{(x-\widehat{P_j})^i}\}_{ij}$, with the entries of $M$ lying in $\mathcal{O}_a$. Write:
$$\det(1-Ms) = 1+\sum_{k=1}^\infty C_ks^k\in\mathcal{O}_a[[s]],$$
so that
\begin{align}\label{Ck_formula}
C_k = \sum_{\substack{S\subseteq\mathbb{Z}_{\geq 0}\times\{1,\cdots,\ell\}\\ |S|=k}}\sum_{\sigma\in\Sym(S)}\sgn\sigma\prod_{(i,j)\in S} D_{(i,j),\sigma(i,j)}.
\end{align}

Following the proof of Theorem 7.2 in \cite{Schmidt1}, line (\ref{Ck_formula}) along with Corollary~\ref{corollary:D_ijnk} yields the Hodge bound with slopes: 	
\[ 
		\frac{a(p-1)(n-1)}{2}\cdot \frac{1}{d_{i_j,j}p^{m-i_j+1}(p-1)} = \frac{a(n-1)}{2d_{i_j,j}p^{m-i_j-1}}.
	\]
Moreover, noting when equality holds in Corollary~\ref{corollary:D_ijnk} implies that if this Hodge bound is obtained, then $d_{i_k}|(p-1)$ for all $1\leq k\leq \ell$. Oppositely, if $d_{i_k}|(p-1)$ for all $1\leq k\leq\ell$, then the Hodge bound is necessarily achieved under the assumption that $d_{i_jj}p^{m-i_j-1}=d_{i_jj}$, because $\frac{n(p-1)}{d_{i_jj}}\leq p-1$ and $u_{\frac{n(p-1)}{d_{i_jj}}}\in\Zptimes$.  

\begin{remark}[The ``Truncated'' Hodge Polygon]

If we use the Dwork theory developed in Section~\ref{Dwork_coh} utilizing the truncated Artin-Hasse exponential, we can derive lower bounds for the Newton polygon of $L_f(s)$, just like we would to compute the Hodge polygon in the traditional case. However, we see that the the resulting lower bound is actually lower than the Hodge bound.

Recall from Proposition~\ref{alpha_cc} that  
	\[
		\ord(F_{nj})\geq \frac{n}{d_{i_j,j}}\cdot\frac{p-(m-i_j)}{p^{m-i_j+1}} = n\cdot (\frac{p-(m-i_j)}{d_{i_j,j}p^{m-i_j+1}}) = n\cdot M_j.
	\]	 
Applying the same procedure as above yields a lower polygon consisting of slopes 
\begin{align}\label{slopes}
		\frac{a(p-1)(n-1)}{1}\cdot \frac{p-(m-i_j)}{d_{i_j,j}p^{m-i_j+1}} = \frac{a(n-1)}{2d_{i_j,j}p^{m-i_j-1}}\cdot \frac{(p-1)(p-(m-i_j))}{p^2}.
\end{align}
Interestingly, \ref{slopes} implies that as $p\to\infty$, this ``truncated'' Hodge polygon converges upward towards the regular Hodge bound. (This is expected since the coefficients of the truncated Artin-Hasse functions converge upon the coefficients of the classical Artin-Hasse as $p\to\infty$, i.e. $u_{ki}\to u_i$ as $p\to\infty$.)

\end{remark}

\end{document}